\documentclass{amsart}

\usepackage{amsmath,amsfonts,amsthm,amstext,amssymb,amscd,amsbsy}
\usepackage[latin1]{inputenc}

\usepackage{indentfirst} 
\usepackage[dvips]{color}
\usepackage{color}
\usepackage{epsfig}
\usepackage{hyperref}

\newtheorem{thm}{Theorem}
\newtheorem{prop}{Proposition}
\newtheorem{lem}{Lemma}
\newtheorem{rem}{Remark}

\newcommand{\rn}[1]{\mathbb{R}^{#1}}

\newcommand{\gotico}{\mathfrak}
\newcommand{\Ad}{\mbox{Ad}}

\def\bq{\begin{equation}}
\def\eq{\end{equation}}

\title[Global behavior of the Ricci flow on homogeneous manifolds]{Global behavior of the Ricci flow on homogeneous manifolds with two isotropy summands.  }
\thanks{This research was partially supported by FAPESP grant 2007/05215-4 (Martins) and by CNPq grant 140431/2009-8 (Grama).}

\author{Lino Grama}
\email[L. Grama]{linograma@gmail.com}

\author[Ricardo M. Martins]{Ricardo Miranda Martins}
\email[R.M. Martins]{rmiranda@ime.unicamp.br}

\address{Department of Mathematics, Institute of Mathematics, Statistics and Scientific Computing. P.O. Box 6065, University of Campinas - UNICAMP.}

\begin{document}

\begin{abstract}
In this paper we study the global behavior of the Ricci flow equation for two classes of homogeneous manifolds with two isotropy summands. Using methods of the qualitative theory of differential equations, we present the global phase portrait of such systems and derive some geometrical consequences on the structure of such manifolds under the action of the Ricci flow.
\end{abstract}

\maketitle

\section{Introduction}
\label{intro}

The Ricci flow of left-invariant metrics on homogeneous has been investigated intensively in the last years, and many papers have been published about these systems, for instance \cite{lauret} and \cite{ds3}. The key point for investigating homogeneous manifolds is that they provide examples of interesting phenomena, see for example \cite{helga} and \cite{kob}.

The Ricci flow equation was introduced by Hamilton in 1982 (see \cite{hamilton}), and is defined by
\begin{equation}
\label{eqnricci}
\frac{\partial g(t)}{\partial t}=-2Ric(g(t)),
\end{equation}
where $Ric(g)$ is the Ricci tensor of the Riemannian metric $g$. The solution of this equation, the so called Ricci flow, is a $1$-parameter family of metrics $g(t)$ in $M$. Intuitively, this is the heat equation for the metric $g$.

The Ricci flow equation (\ref{eqnricci}) for a arbitrary manifold is a nonlinear system of PDEs. However, when restricted to the set of invariant metrics, system (\ref{eqnricci}) reduces to a autonomous nonlinear system of ODEs.

Because of this, it is natural to proceed the study of the Ricci flow from a qualitative point of view, using tools from the theory of dynamical systems. Some recent papers that use this approach are \cite{ds1}, \cite{novosd1}, \cite{ds3} and \cite{ds4}.

The main aim of this paper is to study the equations of the Ricci flow of left-invariant metrics in homogeneous manifolds with two isotropy summands. { A complete list of such manifolds is given in  \cite{arva3}, but we restrict ourselves to two special classes. Specifically, we are interested in the equation (\ref{eqnricci}) restricted to the set of the left-invariant metrics in the generalized flag manifolds} $\displaystyle \frac{SO\left(2n+1\right)}{U(m)\times SO(2k+1)}$ and $\displaystyle \frac{Sp(n)}{U(m)\times Sp(k)}$ (referred sometimes as manifolds of type I and type II, respectively). This study can be carried over to the other classes of generalized flag manifolds with two isotropy summands as well, using the same techniques we develop here, and the results obtained are very similar.

In the cases we want to study, the Ricci flow equation (\ref{eqnricci}) is equivalent to a polynomial differential system. We can study the behavior of such systems at infinity using a method introduced by Poincar\'e, the so called Poincar\'e Compactification. This method allow us to study global phase portraits for polynomial systems.

This paper is organized as follows. In Section 2 we provide a geometrical description of the homogeneous manifolds of types I and II. In Section 3, we describe the Poincar\'e compactification method. In Section 4 we study the dynamics of the Ricci flow equation at infinity for manifolds of type I and II, and in Section 5 we provide a geometrical interpretation for the results obtained in the previous section.

\section{Geometric description of the homogeneous manifold $SO\left(2n+1\right)/\left(U(m)\times SO(2k+1)\right)$ and $Sp(n)/\left(U(m)\times Sp(k)\right)$}
Let $G$ be a compact, connect and semisimple Lie group and denote by $\gotico{g}$ the Lie algebra of $G$. Let $H$ be a closed connected subgroup of $G$, with Lie algebra $\gotico{h}$, and consider the homogeneous space $G/H$. The point $o=eH$ is called the origin of the homogeneous space. 

Since $G$ is compact and semisimple, the Cartan-Killing form of $\gotico{g}$ is non-degenerate and negative definite and we will denote by $Q$ the negative of the Cartan-Killing form. In this case, we say that the homogeneous space $G/H$ is reductive, that is, $$
\gotico{g}=\gotico{h}\oplus \gotico{m} \mbox{\,\,\, and \,\,\,} \Ad(H)\gotico{m}\subset \gotico{m},$$ where $\gotico{m}=\gotico{h}^{\perp}$.  

We naturally identify the tangent space at origin $o$ with $\gotico{m}$. We define the isotropy representation $$j:H \rightarrow \mbox{GL}(\gotico{m})$$ given by $j(h)=\Ad (h)\big|_{\gotico{m}}$ for $h\in H$; in this way, $\gotico{m}$ is a $\Ad(H)$-module. 

A Riemannian metric in $G/H$ is left-invariant if the diffeomorphism $L_a:G/H \rightarrow G/H$ given by $L_a(gH)=agH$ is an isometry for all $a\in G$. A left-invariant metric is completely determined by its value at the origin $o$. 

If ${G}/{H}$ is reductive with an $\Ad(H)$-invariant decomposition $\gotico{g}=\gotico{h}\oplus\gotico{m}$, then there is a natural one-to-one correspondence between the $G$-invariant Riemannian metrics $g$ on ${G}/{H}$ and the $\Ad(H)$-invariant scalar product $B$ on $\gotico{m}$, see \cite{kob}.

An important class of homogeneous spaces are the {generalized flag manifolds (or K\"ahler C-space), which are orbits of an element $X\in \gotico{g}$ by the adjoint action $Ad:G\rightarrow\gotico{g}$ from a compact semisimple Lie group $G$ to $\gotico{g}$ . A classification of the generalized flag manifolds is given in \cite{arva1}.

In this work, we consider the generalized flag manifolds of types I and II, that is, \[\frac{SO\left(2n+1\right)}{U(m)\times SO(2k+1)} \mbox{\,\,\,and\,\,\,} \frac{Sp(n)}{U(m)\times Sp(k)}.\]

Fix a reductive decomposition of the Lie algebra $\gotico{g}$, for $\gotico{g}=\gotico{so}\left(2n+1\right)$ or $\gotico{g}=\gotico{sp}(n)$. In these manifolds, the tangent space at origin $o$ splits into two irreducible non-equivalent $\Ad(H)$-submodules. For a complete list of the generalized flag manifolds with this property, see \cite{arva3}.

Let $g$ be an invariant metric and $B$ the $\Ad$-invariant scalar product on $\gotico{m}$ corresponding to $g$. Then $B$ is given by $B(X,Y)=Q(\Lambda X,Y)$, where the linear operator $\Lambda:\gotico{m}\rightarrow\gotico{m}$ is symmetric and positive with respect to the Cartan-Killing form of $\gotico{g}$. We will denote such an invariant metric $g$ by $\Lambda$ . 

Let $\gotico{m}=\gotico{m}_1 \oplus \gotico{m}_2$ be a decomposition of $\gotico{m}$ into irreducible non-equivalent $\Ad(H)$-submodules. A consequence of the Schur's lemma is that \mbox{$\Lambda\big|_{\gotico{m}_i}=\lambda_i\cdot{Id}\big|_{\gotico{m}_i}$} for $i=1,2$ and therefore any invariant scalar product has the form \[B(X,Y)=\lambda_{1} \cdot Q(X,Y)\big|_{\gotico{m}_{1}}\oplus \lambda_{2} \cdot Q(X,Y)\big|_{\gotico{m}_{2}}.\] For more details about the decomposition of the isotropy representation, see \cite{ale1} and \cite{arva1}.

The Ricci tensor of the an $G-$invariant metric is also an $G-$invariant tensor and is completely determined by its value at the origin $o$. 

\begin{prop}
The components of the Ricci tensor of the manifolds of type I and II 
are given, respectively, by
\bq\label{77tf}\
\left\{
\begin{array}{rcl}
\vspace{0.5cm}\displaystyle r_1&=&\displaystyle-\frac{2(m-1)}{2n-1}-\frac{1+2k}{2(2n-1)}\frac{\lambda_1^2}{\lambda_2^2},\\
\displaystyle r_2&=&\displaystyle -\frac{n+k}{2n-1}-\frac{m-1}{2(2n-1)}\frac{\left(\lambda_2^2-(\lambda_1-\lambda_2)^2\right)}{\lambda_1\lambda_2},
\end{array}
\right.
\eq
and 
\bq\label{77tfu}
\left\{
\begin{array}{rcl}
\vspace{0.5cm}\displaystyle r_1&=&\displaystyle -\frac{2+2m}{2n+2}-\left(\frac{2k}{4n+4}\right)\frac{\lambda_1^2}{\lambda_2^2},\\
\displaystyle r_2&=&\displaystyle -\frac{4m+4k+3}{4n+4}+\left(\frac{4m+2}{16n+16}\right)\frac{ \lambda_1}{ \lambda_2}.
\end{array}
\right.
\eq
\end{prop}
\begin{proof}
A direct computation using Proposition 3 of \cite{arva2}. For more details, see \cite{neiton}.
\end{proof}

The Ricci flow equation for the left-invariant metrics is given by \bq\label{r4originalLINO}\dot \lambda_{i}=-2r_{i}, \ i=1,2.\eq

\section{Poincar\'e Compactification}

The main idea of the Poincar\'e Compactification is to pass from the study of a vector field in a noncompact manifold to the study of a vector field in a compact one, the sphere. This allow us to better understand its behavior in the infinity. This method dates back to 1881, with Poincar\'e. He was studying the behavior of polynomial planar vector fields at infinity by means of the central projection. We recommend \cite{velasco} for the detailed description of this method, including a $n$-dimensional version.

Consider the polynomial differential system \bq\label{ppp1}\left\{\begin{array}{rcl}
\dot x_1&=&P_1(x_1,x_2),\\
\dot x_2&=&P_2(x_1,x_2),\\
\end{array}\right.\eq
with the associated vector field $X=(P_1,P_2)$. The degree of $X$ is defined as $d=\max \{\deg(P_1),\deg(P_2)\}$.

Let \[S^2=\{y=(y_1,y_2,y_3)\in\rn{3}; ||y||=1\}\] be the unit sphere with north hemisphere $S^2_+=\{y\in S^2; y_{3}>0\}$, south hemisphere $S^2_-=\{y\in S^2; y_{3}<0\}$ and equator $S^2_0=\{y\in S^2; y_{3}=0\}$.

Consider the central projections $f_+:\rn{2}\rightarrow S^2_+$ and $f_-:\rn{2}\rightarrow S^2_-$ given by $f_+(x)=\frac{1}{\Delta(x)}(x_1,x_2,1)$ and $f_-(x)=-\frac{1}{\Delta(x)}(x_1,x_2,1),$ where $\Delta(x)=\sqrt{1+x_1^2+x_2^2}$. We shall use coordinates $y=(y_1,y_2,y_3)$ for a point $y\in S^2$.

As $f_+$ and $f_-$ are homeomorphisms, we can identify $\rn{2}$ with both $S^2_+$ and $S^2_-$. The maps $f_+$ and $f_-$ define two copies of $X$, $Df_+(x)X(x)$ in the north hemisphere, based on $f_+(x)$, and $Df_-(x)X(x)$ in the south hemisphere, based on $f_-(x)$. Note that, for $x\in\rn{2}$, when $||x||\rightarrow\infty$, $f_+(x),f_-(x)\rightarrow S^2_0$. This allow us to identify $S^2_0$ with the infinity of $\rn{2}$. Denote by $\overline{X}$ the vector field on $S^2\setminus S^2_0=S^2_+\cup S^2_-$.

To extend $\overline{X}(y)$ to the sphere $S^2$, we define the Poincar\'e compactification of $X$ as\[p(X)(y)=y_{3}^{d-1}\overline{X}(y).\]

\begin{thm}[Poincar\'e, \cite{velasco}]\label{tpv}The vector field $p(X)$ extends $\overline{X}$ analytically to the whole sphere, and in such a way that the equator is invariant.
\end{thm}

If we know the behavior of $p(X)$ around the equator $S^2_0\equiv S_1$, then we know the behavior of $X$ in the neighborhood of the infinity. The natural projection $\pi$ of $S^2$ on $y_3=0$ is called the Poincar\'e disc, and it is denoted by $D^2$. If we understand the dynamics of $p(X)$ on $D^2$, then we completely understand the dynamics of $X$, including at the infinity.

We remark that Theorem \ref{tpv} works just for polynomial vector fields.

Now consider the coordinate neighborhoods $U_i=\{y\in S^2,\, y_i>0\}$ and $V_i=\{y\in S^2,\, y_i<0\}$, for $i=1,2,3$, and the corresponding coordinate maps $\phi_i:U_i\rightarrow \rn{2}$ and $\psi_i:V_i\rightarrow \rn{2}$, given by \[\phi_i(y_1,y_2,y_3)=\left(\frac{y_j}{y_i},\frac{y_k}{y_i}\right), \,\,\psi_i(y_1,y_2,y_3)=\left(\frac{y_j}{y_i},\frac{y_k}{y_i}\right),\] for $i,j,k\in\{1,2,3\}$ with $j<k$. Denote $z=(z_1,z_2)$ the value of $\phi_i(y)$ or $\psi_i(y)$, according to the local chart that is being used. Then we have the following expressions for $p(X)$, written in local charts:

\bq\label{chartU1}U_1: \,\, \frac{z_2^d}{(\Delta(z))^{d-1}}\left(-z_1P_1\left(\frac{1}{z_2},\frac{z_1}{z_2}\right)+P_2\left(\frac{1}{z_2},\frac{z_1}{z_2}\right),-z_2P_1\left(\frac{1}{z_2},\frac{z_1}{z_2}\right)\right);\eq 

\bq\label{chartU2}U_2: \,\, \frac{z_2^d}{(\Delta(z))^{d-1}}\left(-z_1P_2\left(\frac{z_1}{z_2},\frac{1}{z_2}\right)+P_1\left(\frac{z_1}{z_2},\frac{1}{z_2}\right),-z_2P_2\left(\frac{z_1}{z_2},\frac{1}{z_2}\right)\right);\eq

\bq\label{chartU3}U_3: \,\, \frac{1}{(\Delta(z))^{d-1}}\left(P_1(z_1,z_2),P_2(z_1,z_2)\right).\eq 

For the charts $V_1$, $V_2$ and $V_3$, we obtain the same expressions (\ref{chartU1}), (\ref{chartU2}) and (\ref{chartU3}), now multiplied by $(-1)^{d-1}$.

We can avoid the factor $\frac{1}{(\Delta(z))^{d-1}}$ in the local expressions of $p(X)$. In this way, the expression of $p(X)$ is polynomial in each local chart. Note that the singularities at infinity have $z_2=0$.

In the next section, we shall use the Poincar\'e compactification to study the Ricci flow equation (\ref{r4originalLINO}) with $r_1,r_2$ given by (\ref{77tf}) and (\ref{77tfu}). Our aim is to analyse the global behavior of these differential system in the Poincar\'e disc.

\section{Dynamics of the Ricci flow}

The analysis will be done in separate subsections, with a detailed exposition just for the type I manifolds. For type II manifolds, we just present the main result, as the proofs are very similar.

\subsection{Case I: $\displaystyle \frac{SO\left(2n+1\right)}{U(m)\times SO(2k+1)}$}
Consider the system

\bq\label{tipoA}
\left\{
\begin{array}{rcl}
\vspace{0.3cm}\displaystyle \dot x&=&\displaystyle \frac{2(m-1)}{2n-1}+\frac{1+2k}{2(2n-1)}\frac{x^2}{y^2},\\
\displaystyle \dot y&=&\displaystyle \frac{n+k}{2n-1}+\frac{m-1}{2(2n-1)}\frac{\left(y^2-(x-y)^2\right)}{xy},
\end{array}
\right.
\eq
where $n=m+k$, $m>1$ and $k\neq 1$. This is the Ricci flow equation for the flag manifold \[\frac{SO(2n+1)}{U(m)\times SO(2k+1)}, \,\, n=m+k.\]

\begin{lem}\label{invlinet1}System (\ref{tipoA}) have two invariant lines: $\displaystyle \gamma_1(t)=\left(\frac{2(m-1)}{m+2k}t,t\right)$ and $\displaystyle \gamma_2(t)=\left(2t,t\right)$.
\end{lem}
\begin{proof}Let $X(x,y)$ be the vector field associated to system (\ref{tipoA}) and $\gamma(t)=t(a,b)$ be an invariant line for system (\ref{tipoA}). Suppose $a\neq 0$. Solving the equation $\displaystyle X(a,b)\cdot \left(-\frac ba,1\right)=0$ for $a,b$ we obtain the desidered result. The case $a=0$ has no solution.
\end{proof}


\begin{lem}\label{99fkfj}Let $f_+$ be the central projection, $\pi:S^2\rightarrow D^2$ the natural projection and $\gamma_1,\gamma_2$ as in the Lemma \ref{invlinet1}. Then
\[\pi \circ f_+ \circ \gamma_1(t)=\left(\frac{2(m-1)t}{\rho(t)(m+2k)},\frac{t}{\rho(t)}\right), \, \pi \circ f_+ \circ \gamma_2(t)=\left(\frac{2t}{\sqrt{5t^2+1}},\frac{t}{\sqrt{5t^2+1}}\right),\]
where
\[\rho(t)=\sqrt{\frac{{m}^{2}+4\,mk+4\,{k}^{2}+5\,{t}^{2}{m}^{2}-8\,{t}^{2}m+4\,{t}^{2}+4\,{
t}^{2}mk+4\,{t}^{2}{k}^{2}
}{ \left( m+2\,k \right) ^{2}}},\]
are invariant curves for $p(X)$ on $D^2$.
\end{lem}
\begin{proof}This proof is a straightforward calculation.
\end{proof}

We can not apply the Poincar\'e compactification directly to the system (\ref{tipoA}), because this system is not polynomial. However if we multiply it by $\rho=y^2$, obtaining the system
\bq\label{tipoAlinha}
\left\{
\begin{array}{rcl}
\vspace{0.2cm}\displaystyle \dot x&=&\displaystyle \left(\frac{2+4k}{8n-4}\right)x^2+\left(\frac{2m-2}{2n-1}\right)y^2,\\
\displaystyle \dot y&=&\displaystyle \left(\frac{2-2m}{8n-4}\right)xy+\left(\frac{n+k+m-1}{2n-1}\right)y^2,
\end{array}
\right.
\eq
then (\ref{tipoA}) and (\ref{tipoAlinha}) are equivalent in the first quadrant (as $\rho>0$ for $x,y>0$), and we can apply the Poincar\'e compactification to system (\ref{tipoAlinha}). Recall that, for our geometrical analysis, we have just to study the dynamics of (\ref{tipoA}) in the first quadrant. Note also that the line $\gamma_3(t)=(t,0)$ is now invariant for (\ref{tipoAlinha}).

To study the singularities at infinity of (\ref{tipoAlinha}), we should write this system in the local charts of the Poincar\'e compactification. Note that, as $d=2$ for the system (\ref{tipoAlinha}), the expression of $p(X)$ int the charts $U_1,U_2,U_3$ and $V_1,V_2,V_3$ are the same, except by a minus sign.\\

\noindent\textbullet \ Chart $U_1$:
\bq\label{tipoAu1}
\left\{
\begin{array}{rcl}
\vspace{0.4cm}\displaystyle \dot z_1&=&\displaystyle -\frac 12\,{\frac { \left( m+2\,k \right) z_{{1}}}{2\,n-1}}+\frac12\,{\frac {
 \left( 2\,k+2\,n+2\,m-2 \right) {z_{{1}}}^{2}}{2\,n-1}}+\frac12\,{\frac {
 \left( 4-4\,m \right) {z_{{1}}}^{3}}{2\,n-1}},\\
\displaystyle \dot z_2&=&\displaystyle -\frac12\,{\frac {z_{{2}} \left( 1+2\,k \right) }{2\,n-1}}-\frac12\,{\frac {
 \left( 4\,m-4 \right) {z_{{1}}}^{2}z_{{2}}}{2\,n-1}};
\end{array}
\right.
\eq

\medskip

The singularities at infinity (that is, with $z_2=0$) of (\ref{tipoAu1}), with their local behavior are:\\

\noindent $\displaystyle p_1=\left(\frac12\frac{m+2k}{m-1},0\right)$: stable node; $\displaystyle p_2=\left(\frac12,0\right)$: saddle; $p_3=\left(0,0\right)$: stable node.

\medskip

\noindent\textbullet \ Chart $U_2$:
\bq\label{tipoAu2}
\left\{
\begin{array}{rcl}
\vspace{0.4cm}\displaystyle \dot z_1&=&\displaystyle \frac12\,{\frac { \left( -m-2\,k \right) z_{{1}}}{2\,n-1}}+\frac12\,{\frac {
 \left( 2\,k+2\,n+2\,m-2 \right) {z_{{1}}}^{2}}{2\,n-1}}+\frac12\,{\frac {
 \left( -4\,m+4 \right) {z_{{1}}}^{3}}{2\,n-1}},\\
\displaystyle \dot z_2&=&\displaystyle -\frac12\,{\frac {z_{{2}} \left( 1+2\,k \right) }{2\,n-1}}-\frac12\,{\frac {
 \left( 4\,m-4 \right) {z_{{1}}}^{2}z_{{2}}}{2\,n-1}}.
\end{array}
\right.
\eq

\medskip

The singularities at infinity of (\ref{tipoAu2}) with their local behavior are:\\

\noindent $q_1=\left(2,0\right)\,(\equiv p_2)$: saddle; $\displaystyle q_2=\left(\frac{2(m-1)}{m+2k},0\right)\,(\equiv p_1)$: stable node.\\

\medskip

Using the coordinates maps $\phi_i,\psi_i$, $i=1,2,3$, we obtain that these singularities corresponds to the following points in $S^2$:\\

\noindent\vspace{0.2cm}$\displaystyle p_{1}=\left(\frac{2(m-1)}{\sqrt{5m^2-8m+4mk+4k^2+4}},\frac{(m+2k)}{\sqrt{5m^2-8m+4mk+4k^2+4}},0\right)$ (stable node);\\
$\displaystyle p_{2}= \left(\frac25\sqrt{5},\frac15\sqrt{5},0\right)$ (saddle);\\
$\displaystyle p_{3}= (1,0,0)$ (stable node).\\

\begin{lem}The invariant lines $f_+\circ\gamma_1,f_+\circ\gamma_2$ and $f_+\circ\gamma_3$ satisfy the following:\\
(i) $f_+(\gamma_1(0))=f_+(\gamma_2(0))=f_+(\gamma_3(0))=0$;\\
(ii) $\lim_{t\rightarrow\infty}f_+(\gamma_1(t))=p_1$;\\
(iii) $\lim_{t\rightarrow\infty}f_+(\gamma_2(t))= p_2$;\\
(iv) $\lim_{t\rightarrow\infty}f_+(\gamma_3(t))=p_3$.
\end{lem}
\begin{proof}The proof follows from taking limits in the expression in Lemma \ref{99fkfj}.
\end{proof}

In Figure \ref{fig:plotTipo1} we show the singularities and the invariant lines of $p(X)$ on $D^2$. We shall restrict ourselves to the first quadrant.

\begin{figure}[h!]
\includegraphics{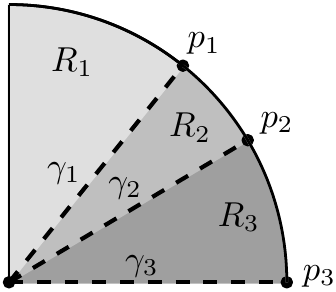}
\caption{Phase portrait for $p(X)$ on $D^2$.}
\label{fig:plotTipo1}
\end{figure}

The following result is the main theorem for type I manifolds.

\begin{thm}\label{mth1}Let $R_1,R_2$ and $R_3$ be the distinguished open regions in Figure \ref{fig:plotTipo1} and $\varphi_t$ be the flow of the projection of $p(X)$ to $D^2$.\\
\noindent (a) If $(x_0,y_0)\in R_1\cup R_2\cup\gamma_1$ then $\lim_{t\rightarrow\infty}\varphi_t(x_0,y_0)=p_1$;\\
\noindent (b) If $(x_0,y_0)\in \gamma_2$ then $\lim_{t\rightarrow\infty}\varphi_t(x_0,y_0)=p_2$;\\
\noindent (c) If $(x_0,y_0)\in R_3$ then $\lim_{t\rightarrow\infty}\varphi_t(x_0,y_0)=p_3$.
\end{thm}

\subsection{Case II: $\displaystyle \frac{Sp(n)}{U(m)\times Sp(k)}$}

Consider the system \bq\label{tipoB}
\left\{
\begin{array}{rcl}
\vspace{0.5cm}\displaystyle \dot x&=&\displaystyle \frac{2+2m}{2n+2}+\left(\frac{2k}{4n+4}\right)\frac{x^2}{y^2},\\
\displaystyle \dot y&=&\displaystyle \frac{4m+4k+3}{4n+4}-\left(\frac{4m+2}{16n+16}\right)\frac xy,
\end{array}
\right.
\eq
where $n=m+k$, $m\geq 1$ and $k\geq 3$. This is the Ricci flow equation for the manifold \[\frac{Sp(n)}{U(m)\times Sp(k)}, \,\, n=m+k.\]

The invariant lines in this case are $\displaystyle \gamma_1(t)=\left(t,\frac14\frac{4k+2m+1}{m+1}t\right)$ and $\displaystyle \gamma_2(t)=\left(t,\frac12t\right)$. As in the Case I, put $\gamma_3(t)=(t,0)$.

The singularities of $p(X)$ in the positive quadrant of $D^2$ are $\displaystyle p_1=\left(\frac{4(m+1)}{\rho},\frac{4k+2m+1}{\rho}\right)$ (stable node), where $\rho=\sqrt{20m^2+36m+16k^2+16mk+8k+17}$, $\displaystyle p_2=\left(\frac25\sqrt5,\frac15\sqrt5\right)$ (saddle) and $p_3=(1,0)$ (stable node).

The analogous of Theorem \ref{mth1} for this case is given bellow. Note that the statement of Theorems \ref{mth1} and \ref{t2} are the same, but the elements involved in Theorem \ref{t2} are as defined in this subsection.

\begin{thm}\label{t2}
Let $R_1,R_2$ and $R_3$ be the distinguished open regions in Figure \ref{fig:plotTipo1} and $\varphi_t$ be the flow of the projection of $p(X)$ to $D^2$.\\
\noindent (a) If $(x_0,y_0)\in R_1\cup R_2\cup\gamma_1$ then $\lim_{t\rightarrow\infty}\varphi_t(x_0,y_0)=p_1$;\\
\noindent (b) If $(x_0,y_0)\in \gamma_2$ then $\lim_{t\rightarrow\infty}\varphi_t(x_0,y_0)=p_2$;\\
\noindent (c) If $(x_0,y_0)\in R_3$ then $\lim_{t\rightarrow\infty}\varphi_t(x_0,y_0)=p_3$.
\end{thm}

\section{Geometrical properties of the Ricci flow}

In this section we derive some results on the convergence of the Ricci flow in the Region $R_3$ (see Theorems \ref{mth1}, \ref{t2} and Figure \ref{fig:plotTipo1}).

We begin with an auxiliar lemma.

\begin{lem}[\cite{itoh}]
\label{lemacolchete} Let $G/H$ be a flag manifold with two isotropy summands. Decompose the Lie algebra of $G$ as $\gotico{g}=\gotico{h}\oplus\gotico{m}_1\oplus\gotico{m}_2$. The the following relations hold: \[[\gotico{m}_1,\gotico{m}_2]\subset \gotico{m}_1, \,\,\, [\gotico{m}_1,\gotico{m}_1]\subset \gotico{m}_2\oplus \gotico{h}, \,\,\, [\gotico{m}_2,\gotico{m}_2] \subset \gotico{h}.\]
\end{lem}

It is well know that a flag manifold is a fiber bundle over a symmetric space (see \cite{burt}). To explicit exhibit this fibration, consider the subalgebra $\gotico{k}=\gotico{h}\oplus \gotico{m}_2$ and let $K$ be the corresponding connected Lie group. Follows from Lemma \ref{lemacolchete} that \[[\gotico{k},\gotico{k}]\subset \gotico{k},\,\,\, [\gotico{k},\gotico{m}_1]\subset\gotico{m}_1 ,\,\,\, [\gotico{m}_1,\gotico{m}_1] \subset\gotico{k}.\] Therefore, the decomposition $\gotico{g}=\gotico{k}\oplus\gotico{m}_1$ is a symmetric decomposition of $\gotico{g}$ and $G/K$ is a symmetric space.

Now we consider the fibration $K/H\cdots G/H \rightarrow G/K$. The base space $G/K$ is {isotropically irreducible}, so any $G$-invariant scalar product in $G/K$ is a multiple of the canonical metric, that is, $B_{base}=\alpha \cdot Q(X,Y)\big|_{\gotico{m}_{1}}$. Therefore, in a invariant metric $\Lambda=(\lambda_1,\lambda_2)$ in $G/H$, $\lambda_1$ is the horizontal part (tangent to the base) and $\lambda_2$ is the vertical part (tangent to the fiber). Using the classification of symmetrical spaces (see \cite{helga}) and comparing the dimensions, we now study each case.
\\

\noindent {\bf Case I:} $G=SO\left(2n+1\right)$, $H=U(m)\times SO(2k+1)$.\\

The dimensions of the irreducible submodules, as calculated in \cite{arva3}, are $$\dim \gotico{m}_1=2m(2k+1)  \,\,\, \mbox{and} \,\,\, \dim \gotico{m}_2=m(m-1);$$ thus, $K=SO(2m)\times SO(2k+1)$ and the corresponding fibration is  \[\frac{SO(2m)}{U(m)}\cdots\frac{SO\left(2n+1\right)}{U(m)\times SO(2k+1)}\longrightarrow \frac{SO\left(2n+1\right)}{SO(2m)\times SO(2k+1)}.\]

Now, from Theorem \ref{mth1} and Figure \ref{fig:plotTipo1}, one can see that $\lambda_1(t)\rightarrow 0$ when $t\rightarrow \infty$, where $\lambda_1(t)$ is the horizontal part of the invariant metric $g(t)$, that evolves under the Ricci flow. In other words, the diameter of the base converges to zero when $t\rightarrow \infty$.\\

\noindent {\bf Case II:} $G=Sp(n)$, $H=U(m)\times Sp(k)$.\\

The dimensions of the irreducible submodules, as calculated in \cite{arva3}, are $$\dim \gotico{m}_1=4mk  \,\,\, \mbox{and} \,\,\, \dim \gotico{m}_2=m(m+1);$$ thus, $K=Sp(m)\times Sp(k)$ and the corresponding fibration is  \[\frac{Sp(m)}{U(m)}\cdots\frac{Sp\left(n\right)}{U(m)\times Sp(k)}\longrightarrow \frac{Sp\left(n\right)}{Sp(m)\times Sp(k)}.\]

Now, from Theorem \ref{t2} and Figure \ref{fig:plotTipo1}, one can see that $\lambda_1(t)\rightarrow 0$ when $t\rightarrow \infty$, where $\lambda_1(t)$ is the horizontal part of the invariant metric $g(t)$, that evolves under the Ricci flow. In other words, the diameter of the base converges to zero when $t\rightarrow \infty$.\\

The following result, independently obtained in \cite{arva3} and \cite{neiton}, characterize the invariant Einstein metrics for the manifolds of types I and II.
\begin{thm}A generalized flag manifold with two isotropy summands admits, up to scale, exactly two invariant Einstein metrics: one K\"ahler and the other non-K\"ahler.
\end{thm}

We summarize Theorems \ref{mth1}, \ref{t2} and the discussion of this section in the following result.

\begin{thm}\label{tfinal}With the notation of Theorems \ref{mth1} and \ref{t2}, let $g_0$ be a invariant metric in the manifolds of type I or II. Remember $R_1,R_2,R_3,\gamma_1,\gamma_2$ in Figure \ref{fig:plotTipo1}. Then the following holds:\\
\noindent (a) If $g_0\in R_1\cup R_2\cup \gamma_1$, then $g_\infty$ is a Einstein (non-K\"ahler) metric.\\
\noindent (b) If $g_0\in\gamma_2$ then $g_\infty$ is a K\"ahler-Einstein metric.\\
\noindent (c) For $g_0\in R_3$, consider the natural fibration from a flag manifold in a symmetric space $G/H\rightarrow G/K$. Then the Ricci flow $g(t)$ with $g(0)=g_0$ evolve in such a way that the diameter of the base of this fibration converges to zero when $t\rightarrow \infty$. 
\end{thm}


\begin{rem}The Einstein metrics described in Theorem \ref{tfinal} are the same obtained in \cite{arva3} and \cite{neiton}.
\end{rem}

\section*{Acknowledgments}

The authors would like to thank Prof. Caio Negreiros for the useful comments.

\end{document}